\documentclass[11pt]{amsart}
\usepackage[utf8]{inputenc}
\usepackage{amsmath}
\usepackage{amsfonts,amssymb}
\usepackage{latexsym}
\usepackage{hyperref}
\usepackage{indentfirst}
\usepackage{inputenc}
\usepackage{yfonts}
\theoremstyle{definition}

\newtheorem{theorem}{Theorem}[section]
\newtheorem{proposition}[theorem]{Proposition}

\newtheorem{corollary}[theorem]{Corollary}
\newtheorem{definition}[theorem]{Definition}

\newtheorem{remark}[theorem]{Remark}
\newtheorem{example}[theorem]{Example}

\newtheorem{open problem}[theorem]{Open problem}

\newcommand{\R}{\mathbb R}

\newcommand{\Z}{\mathbb Z}

\renewcommand{\S}{\mathbb S}

\newcommand{\N}{\mathbb N}
\newcommand{\C}{\mathbb C}

\newcommand{\Dom}{\mathrm{Dom}}

\newcommand{\Spec}{\mathrm{spec}}

\newcommand{\calA}{{\mathcal A}}
\newcommand{\calB}{{\mathcal B}}

\newcommand{\calD}{{\mathcal D}}

\newcommand{\calH}{{\mathcal H}}

\newcommand{\calN}{{\mathcal N}}
\newcommand{\calE}{{\mathcal E}}

\newcounter{para}

\title{On some topics of analysis on noncommutative spaces}
\author{Danila Zaev}
\address{Faculty of Mathematics, Higher School of Economics, Moscow}
\thanks{Supported in part by the Simons Foundation}
\date{\today}
\email{zaev.da@gmail.com}

\begin{document}
\maketitle
\begin{abstract}
	We consider a conservative Markov semigroup on a semi-finite $W^*$-algebra. It is known that under some reasonable assumptions it is enough to determine a kind of differential structure on such a ``noncommutative space''. We construct an analogue of a Riemannian metric in this setting, formulate a Poincar\'e-type inequality, provide existence and uniqueness results for quasi-linear elliptic and parabolic PDEs defined in terms of the constructed noncommutative calculus.
\end{abstract}

\tableofcontents

\section{Introduction}
The main goal of this paper is to establish a framework for the analysis of PDEs in a noncommutative setting. As was shown by J.-L. Sauvageout and F. Cipriani in \cite{CS03}, one can start with a Markov semigroup on a $C^*$-algebra and construct in a natural way a noncommutative version of differential calculus, including square-integrable vector fields, gradient and divergence operators. In commutative case this construction corresponds to some kind of (possibly) non-local analysis, described e.g. in \cite{Hinz13}. In finite-dimensional noncommutative situation it was considered by M.A. Rieffel in \cite{Rieffel}.

``What is a continuous function? It is an element of a commutative $C^*$-algebra''. This is the expression of the idea of the noncommutative analysis approach. Instead of characterizing an object itself, one provides a definition for the algebraic structure of the space of these objects. It is well-known that $C^*$-algebras encode continuous functions (and their noncommutative operator-like analogues) and therefore, by duality, (nice enough) topological spaces;
$W^*$- (or von Neumann) algebras encode bounded measurable functions/operators. It is also possible to define unbounded measurable objects (operators or functions) in the similar manner, spaces of $L^p$- ($p$-absolutely integrable) objects and so on.

The starting point of this paper is the following fascinating fact: it is possible to define abstractly a kind of Sobolev $W^{1,2}$-space (space of square integrable functions with square integrable weak differential), which is called a Dirichlet space. In the following we show that a commutative $C^*$- (or $W^*$-) algebra in a pair with an appropriately defined Dirichlet space acts like a generalized ``Riemannian manifold'', which has a Sobolev (instead of a differentiable one) tangent bundle and a measurable measure-valued ``metric tensor'' defined on it. Moreover, being defined in the ``functional analytic'' way, it can be easily generalized to the noncommutative setting, leading to a beautiful unification of different mathematical concepts.

Our novelty here is a construction of an analogue of a Riemannian metric in the setting of $C^*$-Dirichlet forms. In fact, this concept fills the gap in the informal table of (determining one another) structures below:
\begin{center}
	\begin{tabular}{ | l || l | l |}
		\hline
		 & \textit{is scalar-valued} & \textit{is measure-valued}  \\ \hline
		\textit{defined on functions} & Dirichlet form & Carr\'{e} du Champ form \\ \hline 
		\textit{defined on sections} & Symmetric differential calculus & place for a ``Riemannian metric'' \\ \hline
	\end{tabular}
\end{center}

Then we formulate a Poincar\'e-type inequality in the form appropriate to our goals. With its help we provide existence and uniqueness results for quasi-linear elliptic and some parabolic PDEs defined in terms of the constructed noncommutative calculus. In particular, we formulate a definition of continuity equation on a noncommutative space and deduce an existence result from the general theory of linear evolution PDEs. 

The paper is organized as follows. In Section 2 we review the current state of the area of our interest. Then, in preliminary Section 3 we describe some basic facts about $C^*$- and $W^*$-algebras, weights and traces on them, and the construction of corresponding $L^2$-spaces. We restrict ourselves by considering only tracial case, though not necessarily finite. Note, that necessity for existence of a faithful normal semi-finite trace, that will be assumed throughout the paper, restricts us to the world of semi-finite $W^*$-algebras. In Section 4 we introduce the notions of Markov map, $C_0$-semigroup of Markov operators, its generator, and the key notion of Dirichlet form. The material is not new and rather standard, but our goal is to describe it for a reader not familiar with noncommutative geometry nor operator algebras. Sections 5 and 6 are dedicated to the construction of noncommutative differential calculus associated with a Dirichlet form and, hence, with a Markov semigroup. 

In Section 7 we describe our construction of noncommutative ``Riemannian metric''. In general it takes values in the Banach dual of the underlying $C^*$-algebra. In the case of smooth Riemannian manifold it reduces to a measurable $L^1$-valued bilinear form of $L^2$-sections of the tangent bundle and extends the Riemannian metric tensor defined in the usual sense.

In Section 8 we provide a definition of Poincar\'e-like inequality, which differs from the usual one, but is appropriate for our following study of elliptic PDEs. We consider an example of noncommutative 2-torus equipped with the canonical Dirichlet form and come to a conclusion that it satisfies our version of the inequality.

In Section 9 we define and study linear and quasilinear elliptic PDEs. We mostly follow a standard approach involving Galerkin approximation and Browder-Minty monotonicity trick. Under the standard assumptions this technique allows us to prove existence and uniqueness results of weak solutions.

Finally, in Section 10 we add time axis in our setting and consider parabolic PDEs. In particular, we show that it is possible to give a meaningful formulation of a continuity equation on noncommutative space and provide an existence result for its weak solution.

\section{An introduction to noncommutative analysis}

There are basically two general approaches to noncommutative analysis. The first one generalizes abstract measure theory, and its starting point  is a noncommutative analogue of $L^{\infty}$-space, that is a $W^*$-algebra (it is also called von Neumann algebra, but usually a von Neumann algebra means a $W^*$-algebra with a fixed Hilbert space representation
). Following this approach one defines an $L^1$-space as a Banach predual of a given $W^*$-algebra and construct other $L^p$-spaces using a kind of Banach space interpolation theory. A good reference for that is Section 1 of \cite{Xu07}.

The second approach generalizes topology, and its starting point is a $C^*$-algebra, which is in the commutative case appears to be isomorphic (via a kind of Gelfand duality) to the algebra of all continuous functions vanishing at infinity on some locally compact space. As in the commutative case, given a $C^*$-algebra and an appropriate analogue of a ``reference measure", one can construct all the $L^p$-spaces, $p\in [1,+\infty]$. They share many properties with their classic counterparts, e.g. H\"{o}lder and Minkowski inequalities. This construction is described, for example, in Section 2 of \cite{Alb77}. 

One can generalize the notion of finite measure considering a positive linear functional on a $C^*$-algebra (in a commutative case it corresponds to integration). If this functional has a unit norm, it is called a \textit{state}. Thus, states are elements of the positive part of a unit sphere in the Banach dual of a $C^*$-algebra. On $W^*$-algebras, which are analogues of $L^\infty$-spaces, one is usually interested in \textit{normal} states, which are elements of the space $L^1$ (which is the Banach predual of $L^\infty$). Equivalently, normal states are weak$^*$-continuous states.


The generalization of the notion of non-necessarily finite measure is called a \textit{weight}. It is defined as an unbounded positive linear functional, which is lower semi-continuous and has a dense domain of definition. For the rigorous definition see, for example, Subsection 2.2. of \cite{Kos13} or Section 4 of this paper. A \textit{faithful} weight (state) is such a functional, that the only positive element in its kernel is zero. In the case of commutative $C^*$-algebras, faithful states correspond to Borel probability measures that have full domain. It is natural to choose a ``reference'' weight (state) to be faithful.

The basic dictionary between commutative and noncommutative notions appears to be as follows:
\begin{center}
	\begin{tabular}{ | l | l |}
		\hline
		Commutative & Noncommutative \\ \hline \hline
		Algebra of continuous functions & $C^*$-algebra \\ \hline
		Algebra $L^\infty$& $W^*$-algebra \\ \hline
		Borel measure & Weight on $C^*$-algebra  \\ \hline
		Probability measure & State \\ \hline
	\end{tabular}
\end{center}
and, as we see later, it can be continued in various directions.

An important property of a weight is the \textit{trace} property: $\tau(a\cdot b)=\tau (b \cdot a)$, $\forall a, b \in \calA_+$  (actually, trace of a matrix is a functional of this type on a matrix algebra). If our reference weight has this property, the consequent ``analysis'' is significantly simplified: we are still not too far from commutative world. In the constructions of $L^p$-spaces, described in \cite{Alb77} and \cite{Xu07}, it is assumed that the reference weight has the trace property. For the construction of $L^p$-spaces in the case without this assumption, see Subsection 5 of \cite{Kos13}.

Dirichlet forms in noncommutative setting were originally defined by Albeverio and Hoegh-Krohn in \cite{Alb77}. They started with a $C^*$-algebra and a faithful trace defined on it. Then they consider a quadratic form on a dense domain of the associated $L^2$-space (more precisely, the domain is dense in the ``real" part of the $L^2$-space). If a quadratic form is positive, closed, densely-defined, and has a quadratic Lipschitz contraction property: $\calE(f(a))\leq \|f\|_{Lip_0}^2\cdot \calE(a)$, $\forall f\in Lip_0(\R),\ a\in D(\calE)$, it is called a \textit{Dirichlet form}. If this form is agree with the original $C^*$-algebra in the sense that the intersection of the algebra with the domain of a form appears to be a form core dense in the algebra, it is called \textit{regular}. Dirichlet forms are in one-to-one correspondence with symmetric completely Markov $C_0$-semigroups on $L^2$ (Th. 2.7, 2.8 of \cite{Alb77}).
In addition, the notion of \textit{completely} Dirichlet form and the corresponding notion of \textit{completely} Markov semigroup were introduced in Section 3 of \cite{Alb77}. These notions make sense only in noncommutative context. Theorem 3.2 of \cite{Alb77} establishes a characterization of completely Dirichlet forms, which may serve as a technical tool for the study of forms with this property. 

The research of this topic was continued by B. Davies and M. Lindsay, see \cite{DL92}. In particular, they proved an important technical result: the intersection of the domain of a completely Dirichlet form with the original $C^*$-algebra is an involutive dense subalgebra.

Based on the definitions introduced in \cite{Alb77}, Sauvageout and Cipriani in the series of papers (\cite{Cip98}, \cite{Cip08}, \cite{Cip14}, \cite{CS03}, \cite{CS03-2}, \cite{CS07}, \cite{CS12}, e.t.c.) made a huge breakthrough in the theory and established a bunch of new results. A short overview of their achievements is available in the form of slides (see \cite{Cip14} and \cite{Sav07}). 

Let us mention a few directions in the contemporary research.
\begin{itemize}
	\item Construction of a noncommutative calculus associated to a Dirichlet form. The main paper devoted to this topic is \cite{CS03}. The main result there is an explicit construction of a ``tangent" Hilbert bimodule, which is an analogue of the module of all $L^2$-integrable sections of the tangent bundle of a manifold. It was shown, that there exists a linear map from the domain of the form into this bimodule that satisfies Leibniz property and appears to be an analogue of a ``gradient'' map. Such objects, like Dirichlet forms themselves, their associated Carr\'e du Champ operators $\Gamma$, and infinitesimal generators of the corresponding Markov semigroups can be expressed via the ``gradient'' derivation into the ``tangent'' bimodule.
	
	\item Noncommutative potential theory. It includes the notions of Carr\'e du Champ operators, finite-energy functionals, potentials, and multipliers of Dirichlet spaces appropriate to the noncommutative context. Under the assumption of trace property of the ``reference'' weight this theory was described in \cite{CS12}. An approach without this assumption was also developed: see \cite{CFK12} for the original description of the results and Sections 2 and 3 in \cite{Cip14} for a comprehensive review of the topic.
	
	\item Application to self-similar structures. Using the results about the existence of a ``tangent'' Hilbert bimodule associated to a Dirichlet form, it is possible to construct a potential theory of fractal spaces. The one paper devoted to this direction of research is \cite{Cip08}. See also the works of M. Hinz and A. Teplyaev (e.g. \cite{Hinz12}, \cite{Hinz13-2}).
\end{itemize}

\section{Preliminaries: $C^*$- and $W^*$-algebras, $L^2$-spaces}

\begin{definition}[$C^*$-algebra]
	\begin{enumerate}
		\item An algebra $\calA$ is called \emph{Banach algebra} over field $\mathcal{F}$ iff it is a Banach space over $\mathcal{F}$ such that $\|a\cdot b\|\leq \|a\|\cdot\|b\|$.
		\item An associative Banach algebra $\calA$ over $\C$ is \emph{involutive} Banach algebra iff it is equipped with operation $*: \calA\to \calA$, such that for all $a,b\in \calA$, $t\in \C$
		\begin{enumerate}
			\item $(a^*)^*=a$,
			\item $(ab)^*=a^*b^*$,
			\item $(ta+b)^*=\overline{t}a+b$,
			\item $\|a\|=\|a^*\|$.
		\end{enumerate}
		\item An involutive Banach algebra $\calA$ is a \emph{$C^*$-algebra} iff $\|a a^*\|=\|a\|^2$ for any $a\in \calA$.
	\end{enumerate}
\end{definition}
If $\calA$ is a unital algebra, the last property is equivalent to the following:
$$
\|a\|=\sup\{|\alpha|: \alpha\in \Spec(a)\},
$$
where $\Spec(a)$ is the \emph{spectrum} of $a$: 
$$
\Spec(a):=\{z\in \C: \nexists b\in \calA\ \text{s.t.}\ (a-z\cdot 1_\calA)b=b(a-z\cdot 1_\calA)=1_\calA\}.
$$
The spectrum of any element is a compact nonempty subset of the complex plane.

An element $a\in \calA$ is called \emph{positive} iff $\exists b\in \calA$ s.t. $a=b^* b$. Equivalently, it is an element with a nonnegative real spectrum. 
We denote the set of all positive elements in $\calA$ as $\calA^+$. 
Any $C^*$-algebra is naturally equipped with the following order: $a\geq b \iff a-b\ \text{is positive}$.

An $a\in \calA$ is \emph{self-adjoint} iff $a=a^*$, or, equivalently, $a$ has a real spectrum. The set of all self-adjoint elements is denoted by $\calA^{s.a.}$. In general case $\calA^{s.a.}$ is not closed under the inherited algebraic operation, but appears to be a Banach space over $\R$. Any positive element is automatically self-adjoint: $a=b^*b=(b^*b)^*=a^*$.

For any positive element $a\in \calA^+$ define its square root $a^\frac{1}{2}$ as a positive element $a^\frac{1}{2}\in \calA^+$, such that $(a^\frac{1}{2})^2=a$. The \emph{modulus} of an element $a\in \calA$ is defined as $|a|:=(a^*a)^\frac{1}{2}\in \calA^+$. It is clear, that $\|a\|=\||a|\|$. For a self-adjoint element $a\in \calA^{s.a.}$, the element $|a|-a$ is in $\calA^+$. This fact allows the following decomposition
\begin{equation}
a=a^+-a^-,\ \forall a\in \calA^{s.a.},\ a^+:=|a|,\ a^-:=|a|-a,\ a^+,a^-\in \calA^+.
\end{equation}

\begin{definition}[$W^*$-algebra]
	A $C^*$-algebra $\calN$ is a $W^*$-algebra iff it has a Banach predual space $\calN_*$.
\end{definition}
It is known that any $W^*$-algebra has a unique Banach predual. Thus, along with the $C^*$-norm topology any $W^*$-algebra carry the natural weak$^*$-topology. Here is the essential difference between topology and measure theory: the first is the world of the norm (``supremum'') topology, and the second is the world of the weak$^*$-one. 

There is a set of representation results for $C^*$- and $W^*$-algebras motivating their definitions. Sometimes they are called Gelfand-Naimark type theorems.
\begin{theorem}[Gelfand-Naimark]
	\label{Gelfand-Naimark}
	In the following assertions ``$\simeq$'' is a $^*$-isomorphism.
	\begin{itemize}
		\item
		For every commutative unital $C^*$-algebra $\calA$ there exists a compact Hausdorff space $X$ s.t.
		$$
		\calA\simeq C(X,\C).
		$$
		\item
		For every commutative $C^*$-algebra $\calA$ there exists a locally-compact Hausdorff space $X$ s.t.
		$$
		\calA\simeq C_0(X,\C),
		$$
		where $C_0(X,\C)$ is the uniform completion of the space $C_c(X,\C)$ of all compactly-supported continuous functions.
		\item
		For every $C^*$-algebra ($W^*$-algebra) $\calA$ there exists a complex Hilbert space $H$ and an uniformly (ultraweakly) closed
		$^*$-subalgebra $\pi(\calA)$ of $\calB(H)$ such that
		$$
		\calA \simeq \pi(\calA) \subseteq \calB(H),
		$$
		where $\calB(H)$ is a $W^*$-algebra of all bounded linear operators on $H$.
		\item
		For every commutative $W^*$-algebra $\calA$ there exists a measure space $(X,\calB, \mu)$ s.t.
		$$
		\calA \simeq L^\infty_\C(X,\mu).
		$$
	\end{itemize}
\end{theorem} 

Let $\calA$ be a commutative $C^*$-algebra.
Instead of measures one can equivalently work with ``integrations''. It is a well-known Daniels approach to measure theory (see Section 5.6 of \cite{Kos13}). It is justified to adopt and generalize this approach to noncommutative algebras. It can be easily done for finite measures.
\begin{definition}[State and trace]
	A linear functional $\mu:\calA \rightarrow \C$ is called \emph{state} on $\calA$ iff
	\begin{itemize}
		\item $\mu$ is continuous w.r.t. norm topology on $\calA$,
		\item $\mu$ is positive: $\mu(a)\in [0,+\infty)$ for any $a\in \calA^+$,
		\item $\|\mu\|=1$, where $\|\mu\|:=\sup_{a\in \calA} \frac{|\mu(a)|}{\|a\|}$.
	\end{itemize}
	A state is a \emph{finite trace} iff it has the trace property:
	\begin{itemize}
		\item $\mu(a\cdot b)=\mu(b \cdot a),\ \forall a,b\in \calA$.
	\end{itemize}
\end{definition}
In the case of commutative $C^*$-algebra, there is a bijection between states, finite traces, and Borel probability measures on the Gelfand spectrum of the algebra (which is a compact Hausdorff space). 

In the case of not necessarily finite measure the situation is a bit more complicated. 
\begin{definition}[Weight and trace]
	A map $\mu:\calA^+ \rightarrow \R_\geq\cup {\infty}$ is called \emph{semi-finite weight} on $\calA$ iff
	\begin{itemize}
		\item $\mu(ta+b)=t\mu(a)+\mu(b)$ for $a,b\in A^+$, $t\in \R_\geq$,
		\item $\mu$ is lower-semicontinuous w.r.t. \emph{norm} topology on $\calA^+$,
	\end{itemize}
	A weight is called \emph{semi-finite} iff 
	\begin{itemize}
		\item $\mu$ is finite on a norm-dense subspace: $\Dom(\mu):=\{a: \mu(a)<\infty,\ a\in \calA^+\}$ is dense in $\calA$.
	\end{itemize}
	A weight is a \emph{trace} iff it satisfies the trace property:
	\begin{itemize}
		\item $\mu(a\cdot b)=\mu(b \cdot a),\ \forall a,b\in \calA$.
	\end{itemize}
\end{definition}
If a semi-finite weight is defined on a $W^*$-algebra, it is natural to modify appropriately the definition of semi-finiteness.
\begin{definition}[Normal weight and trace]
	A weight $\mu$ defined on a $W^*$-algebra $\calN$ is called \emph{normal} iff 
	it is lower-semicontinuous w.r.t. weak$^*$-topology on $\calN$. A weight $\mu$ is \emph{normally semi-finite} iff it is \emph{normal} and satisfies the property:
	\begin{itemize}
		\item $\mu$ is finite on a weak$^*$-dense subspace: $\Dom(\mu):=\{a: \mu(a)<\infty,\ a\in \calN^+\}$ is weak$^*$-dense in $\calN$.
	\end{itemize}	
\end{definition}
\begin{example}
	Let $\mu$ on $\Z$ (countable set with discrete topology) be the counting measure, $\calA:=C_0(\Z)$, $\calN:=l_\infty(\Z)=C_b(\Z)$. It can be shown that integration with $\mu$ defines a weight, which is semi-finite on $\calA$, but not semi-finite on $\calN$. Nevertheless, it is normally semi-finite on $\calN$.
\end{example}
Any state $\mu$ can be considered as a semi-finite weight with $\Dom(\mu)=\calA$. 
Since any element of $C^*$-algebra can be represented as a linear combination of two self-adjoint elements: $a=\frac{1}{2}(a^*+a)+\frac{i}{2}(ia^*-ia)$, it follows that for any weight $\mu$,
$$
\mu(a^*)=\overline{\mu(a)},\ a\in \Dom(\mu).
$$
A weight is called \emph{faithful} iff it has the following property:
$$
a\in \calA^+,\ \mu(a)=0 \implies a=0.
$$
In the commutative case any Borel measure corresponds to a faithful weight iff it has a full topological support.

Let us define a $^*$-subalgebra $\mathfrak{m}:=\{a\in \calN: \tau(|a|)<\infty\}$ and a *-ideal $\mathfrak{n}:=\{a\in \calN: \tau(a^*a)<\infty\}$ of $\calN$. By definition, $L^2(\tau)$ is a completion of $\mathfrak{n}$ in the inner product topology: $\langle a, b \rangle_{L^2} := \tau(a^*b)$ for $a,b\in \mathfrak{n}^{(2)}$. The space $L^1(\tau)$ is a completion of $\mathfrak{m}:=\{a\in \calN: \tau(|a|)<\infty\}$ in the norm topology defined as $\|a\|_{L^1}:=\tau(|a|)$. It is clear that $\mathfrak{m}=\calN \cap L^2(\tau)$ is dense in both spaces as well as $\mathfrak{n}=\calN \cap L^1(\tau)$. Let us define positive cones $(L^1(\tau))^+\subset L^1(\tau)$ and $(L^2(\tau))^+\subset L^2(\tau)$ as closures of $\mathfrak{m}^{+}=\{a\in \mathfrak{m}: a\in \calN^+\}$ and $\mathfrak{n}^{+}=\{a\in \mathfrak{n}: a\in \calN^+\}$ respectively. It can be checked that both $L^1(\tau)$ and $L^2(\tau)$ are bimodules over $\calN$ if one defines $a\cdot f$ for $a\in \calN$, $f\in L^1(\tau)$ as a limit of $a \cdot f_n$, $f_n\in \mathfrak{n}$, $f_n\to f$; $f\cdot a$ as a limit of $f_n \cdot a$, $f_n\in \mathfrak{n}$, $f_n\to f$; and analogously for $h\in L^2(\tau)$.

Noncommutative version of the Radon-Nikodym theorem (see section 5.1 of \cite{Kos13}) establishes isomorphism between $L^1(\tau)$ and $\calN_\ast$ via the correspondence: $a(f):=\tau(a\cdot f)=\tau(f\cdot a)$, $a\in \calN$, $f\in L^1(\tau)$. The GNS construction (see section 2.3 of \cite{Kos13}) provides a canonical $^*$-representation of $\calN$ on $L^2(\tau)$ such that $a(h):=a \cdot h$.

\section{Preliminaries: Semigroups and Markovianity}

Let $M_n$ be the algebra of all $n\times n$ matrices, $\mathrm{tr}: M_n\to \C$ be a usual matrix trace, $\calN$ be some semi-finite $W^*$-algebra.

\begin{definition}[Completely positive map]
	A linear map $T: \calN \to \calN$ is \textit{positive} iff $a\in \calN^+$ implies $T(a) \in \calN^+$. A linear map $T: \calN \to \calN$ is called \textit{completely positive} iff for any $n\in \N$ the map
	$Id \otimes T: M_n\otimes \calN \to M_n\otimes \calN$ is positive.
\end{definition}
Any completely positive map is obviously positive. If $\calN$ is commutative, then any positive map is also completely positive. Thus the notions of positivity and complete positivity are coincide in the commutative case. For more information on complete positivity see \cite{Helem}.

\begin{definition}[Markov map]
A linear map $P: \calN \to \calN$ is called Markov iff
\begin{enumerate}
	\item $P$ is completely positive,
	\item $P$ is a contraction: $\|P\|\leq 1$,
	\item $P$ is identity preserving: $P(1_{\calN})=1$.
	Here $1_{\calN}\in \calN$ is an identity element of the algebra.
\end{enumerate}
\end{definition}
It is clear that $\|P\|=1$, and it is also known that $0\leq a\leq 1$ implies $0\leq P(a) \leq 1$ for any Markov map $P$ on $\calN$.

\begin{remark}
	We call an operator $T$ on $\calN$ conservative iff $T(1_{\calN})=1$. Markov maps are always conservative as follows from our definition.
\end{remark}	

Let $\mathrm{Mark}(\calN)$ be the space of all Markov maps on $\calN$.
\begin{definition}[Markov semigroup]
	A family of Markov maps $(P_t)$, $P_t \in \mathrm{Mark}(\calN)$, $t\in [0,\infty)$, is a Markov semigroup iff
\begin{enumerate}
	\item $t \to P_t$ is a morphism of semigroups: $P_{t+s}=P_t+P_s$, $P_0=1$,
	\item it is element-wise weakly-$^*$ continuous: for any fixed elements $a\in \calN$, $f\in \calN_\ast$, $P_t(a)(f)\to a(f)$ as $t\to 0$.
\end{enumerate}
\end{definition}

\begin{definition}[Symmetric trace]
	A normal faithful semi-finite (n.f.s.) trace $\tau$ on $\calN$ is called symmetric w.r.t. Markov semigroup $(P_t)$ iff 
	$$
	\tau(a P_t(b))=\tau(P_t(a) b)
	$$
	for any $a,b\in \calN^+$ and any $t\in[0,\infty)$.
\end{definition}

Let $(P_t)$ be a Markov semigroup on $\calN$, $\tau$ be a symmetric n.f.s. trace. Define a semigroup $(P^{(2)}_t)\subset B(L^2(\tau))$ as $P^{(2)}_t(f)=P_t(f)$ on $\mathfrak{n}$ and extend each map by continuity. The resulting semigroup consists of positivity preserving self-adjoint operators of norm one. The semigroup is in fact strongly continuous w.r.t. $\|\cdot\|_{L^2}$-norm (e.g. Proposition 2.2 of \cite{DL92}), i.e. for any fixed $h\in L^2(\tau)$ $\|P_t^{(2)}(h)-h\|_{L^2}\to 0$ as $t\to 0$.

One can analogously define a semigroup $(P^{(1)}_t)$ on $L^1(\tau)$. Both $(P^{(2)}_t)$ and $(P^{(1)}_t)$ are called Markov semigroups.

\begin{definition}[Generator of Markov semigroup]
	Let $(P_t)$, $(P^{(2)})_t$, $(P^{(1)})_t$ be Markov semigroups on $\calN$, $L^2(\tau)$, $L^1(\tau)$ as defined above. Then their respective generators are uniquely defined by the equalities:
	\begin{eqnarray}
	\label{generator of weak semigroup}
	\lim_{t\searrow 0} \frac{1}{t} (P_t(a)(f)-a(f))-\Delta(a)(f)=0,\; \forall a\in D(\Delta),\ f\in \calN_\ast,\\
	\label{generator of L1 semigroup}
	\lim_{t\searrow 0} \frac{1}{t} \left\|(P^{(1)}_t(a)-a)-\Delta_{(1)}(a)\right\|_{L^1(\tau)}=0,\; \forall a\in D(\Delta_{(1)}),\\
	\label{generator of L2 semigroup}
	\lim_{t\searrow 0} \frac{1}{t} \left\|(P^{(2)}_t(a)-a)-\Delta_{(2)}(a)\right\|_{L^2(\tau)}=0,\; \forall a\in D(\Delta_{(2)}),
	\end{eqnarray}
	where $D(\Delta)\subseteq \calN$, $D(\Delta_{(1)})\subseteq L^1(\tau)$, $D(\Delta_{(2)})\subseteq L^2(\tau)$ are the subspaces of all elements s.t. the corresponding limits in (\ref{generator of weak semigroup}), (\ref{generator of L1 semigroup}), (\ref{generator of L2 semigroup}) exist.
\end{definition}
\begin{proposition}
	$\Delta_{(1)}=\Delta$ when restricted to $D(\Delta)\cap D(\Delta_{(1)})$ and $\Delta_{(2)}=\Delta$ when they are restricted to $D(\Delta)\cap D(\Delta_{(2)})$.
\end{proposition}

It is clear that $\Delta_{(2)}$ is a densely defined unbounded self-adjoint operator on $L^2(\tau)$, which is known to be closed. However, the complete characterization of $L^2$-generators of Markov semigroups is not known for the moment (see \cite{Holevo}). Let $\Delta_{(2)}^{\frac{1}{2}}$ be the unique closed self-adjoint square root of the $\Delta_{(2)}$, $D(\Delta_{(2)}^{\frac{1}{2}})$ be its domain of definition.

\begin{definition}[Dirichlet form of a semigroup]
	Let us fix a $W^*$-algebra $\calN$, Markov semigroup $(P_t)$ and a n.f.s. symmetric trace $\tau$ on $\calN$. Let $\Delta_{(2)}$ be the generator of $(P_t^{(2)})$. Define $D(\calE):=D(\Delta_{(2)}^{\frac{1}{2}})\subseteq L^2(\tau)$. A \textit{Dirichlet form} associated with $(\calN, (P_t), \tau)$ is a closed Hermitian form, uniquely defined on $D(\calE)\otimes D(\calE)$ via the equation:
	$$
	\calE(a,b):=\langle a, \Delta_{(2)}(b) \rangle_{L^2(\tau)},\ \forall a,b\in D(\Delta_{(2)})\subseteq D(\calE).
	$$
	The Dirichlet space is the Hilbert space $D(\calE)$ equipped with the inner product:
	$$
	\langle a, b \rangle_\calE := \langle a, b \rangle_{L^2} + \calE(a, b).
	$$
\end{definition}

\begin{definition}[Abstract Dirichlet form]
	Define an abstract Dirichlet form as a quadratic map $\calE: L^2 \to \R_+\cup {+\infty}$,
	$f \to \calE[f]$ that satisfies the conditions:
	\begin{itemize}
		\item the set $D(\calE):=\{f\in L^2: \calE[f]<\infty\}$ is dense in $L^2$,
		\item $D(\calE)$ is complete in the norm:
		$$
		\|a\|_\calE:=(\|a\|_L^2+\calE[a]))^{\frac{1}{2}},
		$$
		\item the associated Hermitian form $\calE(f,g):=\frac{1}{2}((1+i)(\calE[f]+\calE[g])-\calE[f-g]-i\calE[f-ig])$, defined on $D(\calE)\times D(\calE)$, is real:
		$\calE(a,b)=\calE(a^*, b^*)$ for all $a,b\in \calB:=\calA\cap D(\calE)$,
		\item for all $f\in Lip_0^1(\R)$, $a\in \calB^{s.a.}$
		\begin{equation}
		\calE[f(a)]\leq \calE[a],
		\end{equation}
		where $Lip_0^1(\R)$ is the space of all real 1-Lipschitz functions with a fixed point $0\in \R$.
	\end{itemize}
\end{definition}

For a densely defined Hermitian form  $\calE$ with domain $D(\calE)\subseteq L^2(\tau)$ let $\calE^{(n)}$ denote the Hermitian form on $L^2(\tau)\otimes (M_n, \mathrm{tr})$ given by 
\begin{align}
D(\calE^{(n)})&:=D(\calE)\otimes M_n,\\
\calE^{(n)}(\{a\}_{ij}, \{b\}_{ij})&:=\sum_{i,j}^n{\calE(a_{i,j}, b_{i,j})}.
\end{align}

Let $\calE$ on $L^2(\tau)$ be a Dirichlet form associated with $(\calN, (P_t), \tau)$. Then $\calE^{(n)}$ is an abstract Dirichlet form on $L^2(\tau)\otimes M_n$, $\tau^{(n)}:=\tau\otimes\mathrm{tr}$ for every $n\in \N$. Such forms are called \textit{completely} Dirichlet forms. In fact, symmetric Markovian semigroup are in one-to-one correspondence with completely Dirichlet forms generating conservative semigroup (Theorem 3.3 of \cite{DL92}).

A pair $(\calN, \tau)$ of a $W^*$-algebra and a n.f.s. trace will be called \textit{noncommutative measure space} throughout the paper. The quadriple $(\calN, \tau, \calE, D(\calE))$, where $\calE$ is a completely Dirichlet form generating Markov semigroup, is called (by a slight abuse of notation) by \textit{Dirichlet space}.

\section{Regularity and differential calculus}
Let us fix a triple $(\calN, (P_t), \tau)$ of a $W^*$-algebra, Markov semigroup and a symmetric n.f.s. trace. Let $\calE$ be the associated completely Dirichlet form.
\begin{definition}[Regular subalgebra]
	An algebra $\calA\subseteq \calN$ is called regular iff
	\begin{enumerate}
		\item restriction of $\tau$ on $\calA$ is a faithful, $\|\cdot\|$-semi-finite trace,
		\item $\calA_\tau:=\calA\cap \mathfrak{n}$ is norm dense in $\calA$,
		\item $D(\calE)\cap \calA$ is norm dense in $\calA$,
		\item $D(\calE)\cap \calA$ is dense in $D(\calE)$ w.r.t. the inner product $\langle \cdot, \cdot \rangle_\calE$.
	\end{enumerate} 
\end{definition}
It can be checked, that a $C^*$-subalgebra generated by a regular subalgebra $\calA\subseteq \calN$ is regular itself. The following statement allows one to find a regular $C^*$-subalgebra.

\begin{proposition}[Proposition 2.7 of \cite{DL92}]
	$L^1(\tau)\cap D(\calE) \cap \calN$ is a regular subalgebra.
\end{proposition}

\begin{definition}
	A completely Dirichlet form is called $C^*$-Dirichlet form if a regular $C^*$-subalgebra is given.
\end{definition}

Thus, every completely Dirichlet form is a $C^*$-Dirichlet form for \textit{some} $C^*$-subalgebra. In fact, it is possible to start with a $C^*$-algebra $\calA$ and a faithful semi-finite trace on it, then construct an $L^2(\tau)$-space, the G.N.S. representation of $\calA$ on $L^2(\tau)$, and set $L^\infty(\tau)$ to be a weak$^*$-closure of the representation of $\calA$. It appears that if $\calA=\calN$ is a $W^*$-algebra, $L^\infty(\tau)=\calN$.

\begin{proposition}[Proposition 3.4 of \cite{DL92}, Proposition 2.2 of \cite{CS03}]
	If $\calE$ is a $C^*$-Dirichlet form w.r.t. $C^*$-algebra $\calA$, the space $\calB:=\calA\cap D(\calE)$ is a $*$-algebra w.r.t. multiplication and involution induced from $\calA$.
\end{proposition}

A Hilbert bimodule over $C^*$-algebra $\calA$ is a $*$-representation of $\calA\otimes_{max}\calA^\circ$ on a Hilbert space, where $\otimes_{max}$ is a maximal $C^*$-tensor product and $\calA^\circ$ is an algebra opposite to $\calA$ (the same algebra with reversed order of multiplication).

Let $\calA$ be a $C^*$-algebra, $\tau$ be a faithful semi-finite trace on $\calA$, $\calB\subseteq \calA\cap L^2(\tau)$ be a given $^*$-subalgebra dense in both spaces.
\begin{definition}[Symmetric differential calculus]
	The triple $(\calH, J, \partial)$ is called a symmetric differential calculus iff
	\begin{enumerate}
		\item $\calH$ is a Hilbert bimodule over $\calA$,
		\item $J: \calH \to \calH$ is an antilinear involution,
		$$
		J(a h b)= b^* J(h) a^*,
		$$
		\item $\partial: \calB \to \calH$ is a symmetric derivation:
		$$
		\partial(a^*)=J(\partial(a)),
		$$
		$$
		\partial(ab)=\partial(a)b+a\partial(b),
		$$
		which is \textit{closable} as a linear operator from $L^2(\tau)$ to $\calH$.
	\end{enumerate}
\end{definition}
It follows, in particular, that
\begin{equation}
\|c(a\otimes b)\|_\calH \leq \|c\| \|a\otimes b\|_\calH, \; \; \|(a\otimes b)d\|_\calH \leq \|d\| \|a\otimes b\|_\calH,
\end{equation}
\begin{equation}
\langle h a, g\rangle_\calH = \langle h, g a^*\rangle_\calH, \; \; \langle a h, g\rangle_\calH = \langle h, a^*g\rangle_\calH.
\end{equation}
for $a,b,c\in \calB$, $h,g\in \calH$.

\begin{theorem}[Theorem 8.2 of \cite{CS03}]
	\label{diff calculus from dirichlet}
	For every Dirichlet space $(\calN, \tau, \calE, D(\calE))$ and a regular $C^*$-algebra $\calA$, $\calB:=D(\calE)\cap \calA$ is a $^*$-subalgebra dense in both spaces, and there exists a symmetric differential calculus $(\calH, J, \partial)$ such that
	\begin{itemize}
		\item $\calH$ is the completion of $\calB \otimes \calB / \ker \|\cdot\|_\calH$ with respect to the norm:
		\begin{equation}
		\|a\otimes b\|_\calH = \frac{1}{2}(\calE(a,abb^*)+\calE(abb^*,a)-\calE(bb^*,a^*a)),
		\end{equation}
		\begin{equation}
		\ker \|\cdot\|_\calH:=\left\{\sum_i a_i\otimes b_i \in \calB\otimes \calB: \left\|\sum_i a_i\otimes b_i \right\|_\calH=0 \right\},
		\end{equation}
		\item $\calH$ is equipped with a Hilbert bimodule structure over $\calA$, defined on $\calB$ by the formulas:
		\begin{equation}
		c(a\otimes b):=(c a)\otimes b - c\otimes (ab),
		\end{equation}
		\begin{equation}
		(b\otimes c)a:=b\otimes (ca),
		\end{equation}
		for all $a,b,c\in \calB$.
		\item 
		\begin{eqnarray}
		\partial(a)b=a\otimes b,\\
		\calE(a,b)=\langle \partial a, \partial b \rangle_\calH.
		\end{eqnarray}
	\end{itemize}
\end{theorem}
By abuse of notation an element $a\otimes b\in \calB\otimes \calB$ and its image in $\calH$ will be denoted by the same symbol.

\begin{definition}
	Let $(\calA, \tau)$ and $\calB$ be as above. Define a category of symmetric differential calculi with symmetric differential calculi as objects and bimodule isometric isomorphisms commuting with derivations and involutions as morphisms, i.e. morphisms are maps of the form $T: (\calH_1, J_1, \partial_1) \mapsto (\calH_2, J_2, \partial_2)$
	\begin{align*}
	&T: \calH_1 \mapsto \calH_2,\\
	&T_1(a h b)= a T_2(h) b,\; \forall a,b\in \calA\\
	&T\circ\partial_1=\partial_2,\\
	&J_1\circ T = J_2.
	\end{align*}
\end{definition}

\begin{theorem}[Theorem 8.3 of \cite{CS03}]
	For every Dirichlet space $(\calN, \tau, \calE, D(\calE))$ and a regular $C^*$-algebra $\calA$,
	the symmetric differential calculus $(\calH, J, \partial)$ defined in (\ref{diff calculus from dirichlet}) is an initial object in the full subcategory of all symmetric differential calculus satisfying
	$$
	\calE(a,b)=\langle \partial a, \partial b \rangle.
	$$
\end{theorem}

The derivation $\partial$ constructed in (\ref{diff calculus from dirichlet}) is called the \textit{gradient operator} associated with Dirichlet form $(\calE, D(\calE))$.

Since $\partial$ is a closable operator, it can be extended to a closed one, denoted by the same symbol: $\partial: L^2(\tau) \rightarrow \calH$. It acts between Hilbert spaces, and one is able to define its adjoint: $\partial^*: \calH\rightarrow L^2(\tau)$. Since $\partial$ is thought of as a noncommutative analogue of gradient, it is reasonable to interpret $\partial^*$ as an analogue of divergence. Theorem 8.2 of \cite{CS03} asserts that the generator $\Delta_{(2)}$ of the semigroup associated with a Dirichlet form can be expressed as a composition $\Delta_{(2)}=\partial^*\circ \partial$ on $D(\Delta_{(2)})$, and, hence, can be thought of as a generalization of a Laplace operator.

\subsection{Examples of noncommutative differential calculi}
\subsubsection{Group algebras with length functions}
\label{group algebra example}
Consider a locally-compact group $G$ that is unimodular (it means that the left and the right Haar measures coincide). On the space of all continuous complex-valued function with compact support define multiplication ``$\star$'' as the convolution and involution by the formula: $f^*(g):=\overline{f(g^{-1})},\ f\in C_c(G),\ g\in G$. The norm on $C_c(G)$ is defined as follows: $\|f\|:=\sup\{\|f\star h\|_{L^2(\chi)}: h\in L^2(\chi),\ \|h\|_{L^2(\chi)}=1 \}$. Then one can consider the corresponding reduced $C^*$-algebra $C^*_{r}(G)$, which is the completion of $C_c(G)$ w.r.t. the defined norm. It is commutative iff $G$ is abelian. A reference trace is defined as a continuous extension of $\tau(f):=f(e)$, where $f\in C_c(G)$ and $e\in G$ is the identity element.

Equip $G$ with a continuous length function $l: G\rightarrow [0,+\infty)$ that is conditionally of negative type. For the definition see, for example, Section 2.5 of \cite{Zeng14}, or Section 2.10 of \cite{BHV07}. The corresponding Dirichlet form is defined by the formula: $\calE(f):=\int_G |f(g)|^2 l(g) d\chi_G(g)$. This a conservative regular completely Dirichlet form, as shown in Example 10.2 of \cite{CS03}. This example is also described in \cite{CS12} (Example 2.7 and Example 5.2), where an explicit construction of the ``tangent'' bimodule and the associated gradient derivation is provided.

\subsubsection{Dynamical systems}
When an action of a group is given by homeomorphisms of some topological space, it is possible to construct a convolution $C^*$-algebra of the associated action groupoid. If the group is equipped with some kind of metric data, it is possible to define a Dirichlet space associated with the action and this data. Probably, there is no description of the construction in general case, but there are several well-studied examples. 

The first one is a classic example of noncommutative torus. In fact, a 2-dimensional noncommutative torus is the $C^*$-algebra associated with the action of $\Z$ on $\S^1$ by irrational rotations. The canonical trace and a Dirichlet form is described in Example 2.8 of \cite{CS12}. A description of the associated differential calculus can be found in Example 5.3 of \cite{CS12} and Subsection 10.6 of \cite{CS03}. Some results about concentration are described in Subsection 7.1.2 of \cite{Zeng14}.

Another example of this type, which appears in a physical model, is described in Example 5.4 of \cite{CS12}.

\subsubsection{Riemannian foliation}
J.-L. Sauvageout in \cite{Sav96} described a construction of a transverse heat flow on a Riemannian foliation and the associated noncommutative Dirichlet space.

\subsubsection{Clifford bundles}
A classic example of a noncommutative Dirichlet space is the one associated with a Clifford $C^*$-algebra. See Subsection 10.5 of \cite{CS03} for the definition and the description of the associated ``tangent" bimodule.

More generally, one can consider a Dirichlet space associated to a Clifford bundle of a Riemannian manifold. The references are the paper \cite{CS03-2}, Example 2.6 in \cite{CS12} and Subsection 10.7 of \cite{CS03}.

\section{Carr\'{e} du Champ form}

As in the classic theory of Dirichlet spaces, a finer analysis requires a notion of a ``Carr\'{e} du Champ'' (or gradient) form. In probability theory it is a measure-valued quadratic form with the same domain as the Dirichlet form, satisfying the equality $\calE[a]=\int 1 d\Gamma[a]$. In noncommutative setting $\Gamma$ should have its values in the positive part of the Banach dual of a $C^*$-algebra $\calA$ and should be connected with $\calE$ in the similar manner: $\calE[a]=\langle\Gamma[a], 1_{\calA^{**}}\rangle:=1_{\calA^{**}}(\Gamma[a])$. It is worth to note, that the double dual of a $C^*$-algebra is always a $W^*$-algebra and, hence, has a multiplication identity $1_{\calA^{**}}$.

Let $\calA^*$ be Banach dual space of $C^*$-algebra $\calA$, $\calA^*_+\subset \calA$ be the closed cone of all positive functionals. The space $\calA^*$ can be naturally equipped with the structure of bimodule over $\calA$:
\begin{equation}
(a\cdot_l m) (b):=m(ab),\ a,b\in \calA, m\in \calA^*,
\end{equation}
\begin{equation}
(m\cdot_r c) (b):=m(bc),\ c,b\in \calA, m\in \calA^*.
\end{equation}
Let $\sigma: \calA^* \to \calA^*$, $\sigma(m)(a):=m(a^*)$ for any $a\in \calA$, $m\in \calA^*$. It is easy to note, that
$\sigma (amb)=b^*\sigma(m) a^*$.
\begin{remark}
	When this does not lead to a confusion, we will omit symbols $\cdot_l$ and $\cdot_r$ in the notation of left and right bimodule multiplications.
\end{remark}

Fix some Dirichlet space $(\calN, \tau, \calE, D(\calE))$. Let $\calA$ be some regular $C^*$-subalgebra of $\calN$. Then one can define a Carr\'e du Champ form $\Gamma$ as follows. For $f,g\in D(\calE)$, $a\in \calA$:
$$
\Gamma(f,g)(a):=\langle \partial (f), \partial (g)  a \rangle_\calH,
$$
where $\partial: D(\calE)\to \calH$ is a gradient operator associated with Dirichlet form $\calE$.
Since $\|ah\|_\calH\leq \|a\|_\infty\|h\|_\calH$, a Carr\'e du Champ form takes values in $\calA^*$ (the Banach dual of $\calA$).

One can reformulate inner product on $\calH$ in the terms of $\Gamma$:
\begin{equation}
\langle a\otimes b, c\otimes d\rangle_\calH = 1_{\calA^{**}}(b^*\Gamma(a,c)d).
\end{equation}
\begin{equation}
\left\|\sum_i a_i\otimes b_i\right\|^2_{\calH} = \sum_i \sum_k 1_{\calA^{**}}(b_i^*\Gamma(a_i,a_k)b_k).
\end{equation}
and also the Dirichlet form in terms of $\Gamma$:
\begin{equation}
\calE(a,b)=1_{\calA^{**}}(\Gamma(a,b)).
\end{equation}
Here $1_{\calA^{**}}$ is the identity element of $W^*$-algebra $\calA^{**}$.

It is also possible to explicitly express $\Gamma$ via $\calE$:
\begin{equation}
	\Gamma_\calE[a](b):=\calE(a,ab^*)+\calE(ab^*,a)-\calE(b^*, a^*a),\ \forall a,b\in \calB:=\calA\cap D(\calE).
\end{equation}

See Lemma 9.1 of \cite{CS03} for the proof of the nontrivial part of the following statement.
\begin{proposition}
	\label{properties of CdC}
	A Carr\'e du Champ (CdC) form $\Gamma: L^2(\tau) \to \{\calA^*_+\cup {+\infty}\}$, $f \to \Gamma[f]$, defined above, has the following properties:
	\begin{itemize}
		\item the set $D(\Gamma):=\{f\in L^2: \Gamma[f]\in \calA^*_+\}$ is dense in $L^2$,
		\item $\calB:=D(\Gamma)\cap \calA$ is a $^*-$subalgebra of $\calA$,
		\item $D(\Gamma)$ is complete in the norm:
		$$
		\|a\|_\Gamma^2:=\|a\|_{L^2(\tau)}^2+1_{\calA^{**}}(\Gamma[a]),
		$$ 
		where $1_{\calA^{**}}$ is the identity element of $W^*$-algebra $\calA^{**}$,
		\item the associated $\calA^*$-valued form $\Gamma(f,g):=\frac{1}{2}((1+i)(\Gamma[f]+\Gamma[g])-\Gamma[f-g]-i\Gamma[f-ig])$, defined on $D(\Gamma)\times D(\Gamma)$, is Hermitian over $\C$: $\forall t,s\in \C, \forall a,b\in D(\Gamma)$
		\begin{eqnarray}
		\left.\begin{aligned}
		&\Gamma(ta, sb)=t^*s\Gamma(a,b),\\
		&\Gamma(a,b)=\sigma(\Gamma(b,a)),
		\end{aligned}\right.
		\end{eqnarray} and satisfies
		\item reality:
		$\Gamma(a,b)=\Gamma(a^*, b^*)$ for all $a,b\in \calB\subseteq \calA$,
		\item representability: for all $a,b,c\in \calB$
		\begin{equation}
		\Gamma(ab,c)-\Gamma(b,ac)=b^*\Gamma(a,c)-\Gamma(b,a^*)c,
		\end{equation}
		\item complete positivity: for any $n\in \N$, $\{a_i\}_{i=1}^n\subset \calA$, $\{b_i\}_{i=1}^n\subset \calA$
		\begin{equation}
		\sum_{j,k=1}^n b_j^*\Gamma(a_j, a_k)b_k\in \calA^*_+.
		\end{equation}
	\end{itemize}
\end{proposition}


It is easy to check that an element $f\in \calA^*$ belongs to the subspace $L^1(\tau)\subset \calA^*$ iff it is a continuous functional w.r.t. the topology on $\calA$ induced by the inclusion $\calA\subseteq (L^\infty(\tau), weak^*)$.

In \cite{Zeng14} a Markov semigroup on $L^2(\tau)$ is called a \textit{noncommutative diffusion} semigroup iff the image of the associated form $\Gamma$ is a subset of $L^1(\tau)$ (see Subsection 2.4 of \cite{Zeng14} for details).
In this case it is possible to define $\Gamma(a)$ using the generator $\Delta_{(2)}$ of a noncommutative diffusion semigroup:
$$
2\Gamma[a]:= L(a^*)\cdot a + a^*\cdot L(a) - L(a^*a),\; a\in D(\Delta_{(2)}).
$$ 
It is clear that in this case $\tau(\Gamma[a])=\calE[a]$ for any $a\in D(\calE)$.

Under the assumption that the semigroup associated with a Dirichlet form is a noncommutative diffusion, one can formulate $BE(K, \infty)$ condition for some $K>0$:
\begin{equation}
\label{BE condition}
\Gamma(P_t(a))\leq e^{-2K t} P_t(\Gamma(a)),\ \forall a\in L^1(\tau)\cap D(\calE)\cap \calA,\ \forall t\geq 0,
\end{equation}
It can be checked, that both the right and the left hand sides of this inequality are well-defined.

If we do not assume that the form $\Gamma$ has its values in $L^1(\tau)$, we may face some problems with the right-hand side of (\ref{BE condition}), which can be resolved under additional assumptions. In particular, one need to extend the operator $P_t$ to $\calA^*$ ($\forall\ t\geq 0$). The extension can be defined via duality: $P_t(z)(a):=z(P_t(a)),\ a\in \calA,\ z\in \calA^*$ in the case $P_t$ has the following Feller-type property: $a\in \calA \implies P_t(a)\in \calA$. This property of semigroups on $C^*$-algebras is studied in details in \cite{Sav99}.

\section{Riemannian metric}
It is possible to construct an $\calA^*$-valued sesquilinear mapping on $\calH$ that plays a role of Riemannian metric tensor in noncommutative geometry. In a finite-dimensional situation this idea was explored in details in the paper \cite{Rieffel} of Marc A. Rieffel.

For simple tensors $a\otimes b$, $c\otimes d$ in $\calH$ set
$$
\hat{R}(a\otimes b, c\otimes d):=b^*\Gamma(a,c)d.
$$
The following theorem is a noncommutative generalization of Lemma 2.1 from \cite{Hinz13}.
\begin{theorem}
	$\hat{R}$ can be uniquely extended to a well-defined $\C$-sesquilinear mapping
	$$
	R: \calH \times \calH \to \calA^*
	$$
	such that for any $h \in \calH$, $R[h]\in \calA^*_+$. For any $h,g\in \calH$ one has $1_{\calA^{**}}(R(h,g))=\langle g, h \rangle_\calH$.
\end{theorem}
\begin{proof}
	For any finite linear combination $\sum_i a_i\otimes b_i \in \calB \otimes \calB$ and any $f\in \calN_+$,
\begin{multline*}
\sum_i \sum_k (b_i^*\Gamma(a_i, a_k)b_k) (f)=\sum_i \sum_k \Gamma(a_i, a_k)) (b_i^* f b_k)
=\sum_i \sum_k \Gamma(a_i, a_k)) (b_i^* \sqrt{f}^* \sqrt{f} b_k)=\\
=\sum_i \sum_k 1_{\calA^{**}}((\sqrt{f} b_i)^*\Gamma(a_i, a_k)) \sqrt{f} b_k)=\left\|\sum_i a_i \otimes (\sqrt{f} b_i) \right\|_{\calH}^2\geq 0.
\end{multline*}
Hence $\sum_i \sum_k (b_i^*\Gamma(a_i, a_k)b_k)\in \calA^*_+$ and
$
\|\sum_i a_i\otimes b_i\|=0 \implies \sum_i \sum_k (b_i^*\Gamma(a_i, a_k)b_k)=0
$
as a linear functional. Therefore
$$
R\left[\sum_i a_i\otimes b_i\right]:=\sum_i \sum_k b_i b_k \Gamma(a_i, a_k)
$$
is a well-defined element of $\calA^*_+$. It coincides with $\hat{R}$ on simple tensors.

Let $h\in \calH$ and $(h_k)_{k=1}^\infty \subset \calH$ be a sequence of finite linear combinations of simple tensors 
$h_k=\sum_{i=1}^{n_k} f_i^{(k)}\otimes g_i^{(k)}\in \calH$ approximating $h$ in $\|\cdot\|_\calH$. For any $f\in \calN_+$
\begin{multline*}
\lim_k R(h_k)(f)=\lim_k \sum_i \sum_j ((b_i^{(k)})^*\Gamma(a_i^{(k)}, a_j^{(k)})b_j^{(k)}) (f)=\\
=\lim_k \sum_i \sum_j \Gamma(a_i^{(k)}, a_j^{(k)})) ((b_i^{(k)})^* \sqrt{f}^* \sqrt{f} b_j^{(k)})=\\
=\lim_k \sum_i \sum_j 1_{\calA^{**}}((\sqrt{f}b_i^{(k)})^*\Gamma(a_i^{(k)}, a_j^{(k)})\sqrt{f} b_j^{(k)})
=\lim_k\|h_k\sqrt{f}\|^2_\calH=\|h\sqrt{f}\|^2.
\end{multline*}

Set $R[h](f):=\lim_k R[h_k](f)$ for any $f\in \calN_+$. Since one can decompose any element $f\in \calN$ as a linear combination of positive elements: $f=f_1-f_2+if_3-if_4$, where
$$
f_1:=\left|\frac{x+x^*}{2}\right|,\; f_2:=f_1-\frac{x+x^*}{2},\; f_3:=\left|\frac{x-x^*}{2i}\right|,\; f_4:=f_3-\frac{x-x^*}{2i},\
$$
$R[h](f)$ can be defined on $\calN$ as follows:
\begin{equation}
\label {eq Rhf}
R[h](f):=\lim_k R[h_k](f_1)-\lim_k R[h_k](f_2)+i\lim_k R[h_k](f_3)-i\lim_k R[h_k](f_4).
\end{equation}
It is well-defined continuous linear functional, since
$$
R[h](f)=\|h\sqrt{f_1}\|^2-\|h\sqrt{f_2}\|^2+i\|h\sqrt{f_3}\|^2-i\|h\sqrt{f_4}\|^2,
$$
$$
\left|R(h)(f)\right|\leq 4 \|f\|_\infty \|h \|_\calH 
$$
\end{proof}

It can be noted, that the map $R: \calH\times \calH \to \calA^*$ constructed above satisfies the following properties:
\begin{itemize}
	\item sesquilinear over $\C$: $\forall t,s\in \C, \forall h,g\in \calH$, $R(ta, sb)=t^*sR(a,b)$,
	\item $R(h,g a)=R(h,g)\cdot_r a$
	\item $\sigma(R(h,g))=R(g,h)$,
	\item $R(a h,g)=R(h,a^*g)$,
	\item $\{h\in\calH: R(h,g)=0\ \forall g\in\calH \}=0$.
	\item positivity:
	$R(h,h)\in \calA^*_+$,
\end{itemize}
for all $h,g\in \calH$, $a\in \calA$.
It is a consequence of (\ref{properties of CdC}) and the very definition of $R(\cdot, \cdot)$. See also Section 3 of \cite{Rieffel} and Corollary 2.1 in \cite{Hinz13}.


\section{Poincar\'e inequality}
In this section we provide a definition of Poincar\'e-like inequality, which differs from the usual one, but is appropriate for our following study of PDEs.

Let us define a kernel of a Dirichlet form as follows.
$$
\ker \calE:=\left\{u\in D(\calE): \calE[u]=0\right\}.
$$

Let $D(\calE)$ (which is a Hilbert space) be decomposed into orthogonal sum
$$
D(\calE)=\ker \calE \oplus D_\perp(\calE).
$$
It is clear that $D_\perp(\calE)$ is a Hilbert space w.r.t. inner product $\calE(\cdot, \cdot)$.

\begin{definition}
	$\calE$ is said to satisfy Poincar\'e inequality iff for any $a\in D_\perp(\calE)$, 
	\begin{equation}
	\label{Poincare inequality}
	\|a\|_{L^2}^2\leq C \calE[a],
	\end{equation}
	where $C>0$ is some constant.
	The Poincar\'e constant $C_P$ is defined as the infimum of all constants $C$ such that the inequality (\ref{Poincare inequality}) is satisfied.
\end{definition}
Assume that $\dim(\ker \calE)\leq 1$, i.e. it is either trivial or one-dimensional. Then the dimension of $\ker \Delta_{(2)}$ is also not greater than one. Is such a case the Poincar\'e constant can be interpreted as the inverse of the spectral gap $g$ of $\Delta_{(2)}$.
$$
\frac{1}{C_P}=g:=\{\lambda : \lambda>0,\ \lambda\in \sigma(\Delta_{(2)}) \}>0.
$$

It should be noted, that in some classic definitions of Poincar\'e inequality it is required for $\ker \calE$ to be one-dimensional. We do not assume that and we show that our formulation of the inequality is enough to establish existence/uniqueness results for many PDEs (see the next two sections).

Let us consider a toy noncommutative example.
\begin{example}[noncommutative 2-torus]
	A noncommutative 2-torus can be defined in different equivalent ways. One of them is to define it as a universal $C^*$-algebra generated by two unitaries $U$ and $V$ satisfying the relation $VU=e^{2i\pi\theta}UV$ for some fixed parameter $\theta\in \R$. The unique tracial state $\tau$ is defined by the equality:
	$$
	\tau \left(\sum_{n,m\in\Z} \alpha_{n,m} U^n V^m\right):= \alpha_{0,0},\; n,m\in\Z,\ \alpha_{n,m}\in \C.
	$$
	If one equip noncommutative torus with the Markov semigroup of the form:
	$$
	P_t(U^n V^m)=e^{-t(n^2+m^2)}U^n V^m,\; n,m\in\Z,
	$$
	the associated $C^*$-Dirichlet form is the closure of the form defined by
	$$
	\calE\left[\sum_{n,m\in\Z} \alpha_{n,m} U^n V^m\right]=\sum_{n,m\in\Z} (n^2+m^2)|\alpha_{n,m}|^2
	$$
	on the elements with finite number of non-zero components $\alpha_{n,m}$.
	
	It is clear, that the Poincar\'e inequality is satisfied in this case, and $C_P=1$:
	\begin{equation*}
	\left\|\sum_{n,m\in\Z} \alpha_{n,m} U^n V^m\right\|_{L^2(\tau)}=\sum_{n,m\in\Z} |\alpha_{n,m}|^2\leq \sum_{n,m\in\Z} (n^2+m^2)|\alpha_{n,m}|^2=\calE\left[\sum_{n,m\in\Z} \alpha_{n,m} U^n V^m\right].
	\end{equation*}
	The $\ker \calE$ is one-dimensional and consists of constants. Hence the spectral gap of the corresponding operator $\Delta_{(2)}$ is equal to one.
\end{example}

For the alternative definitions of Poincar\'e inequalities in noncommutative setting see the series of papers of M. Junge and Q. Zeng (\cite{JunZeng13}, \cite{JunZeng13-2}, \cite{Zeng13}, \cite{Zeng14}, e.t.c.), where they study relations between noncommutative analogues of various consentration-like inequalities. For the short review of their results see Subsection 1.2.2 of \cite{Zeng14}.

\section{Quasilinear elliptic equations}
Assume that a Dirichlet space $(\calN, \tau, \calE, D(\calE))$ and a regular $C^*$-algebra $\calA$ are chosen and fixed throughout the section. We also assume $D(\calE)$ to be a separable space.
Let us first prove the existence of a weak solution for a linear Poisson equation:
\begin{equation}
\label{Poisson equation}
\partial^* \partial u = f,\; f\in D(\calE),
\end{equation}
\begin{definition}
	An element $u\in D(\calE)$ is called a weak solution for (\ref{Poisson equation}) iff
	$$
	\langle \partial u, \partial g \rangle_\calH = \langle f, g \rangle_{L^2(\tau)}
	$$
	for any $g\in D(\calE)$.
\end{definition}
\begin{theorem}[Dirichlet principle]
	Let $\calE$ satisfies Poincar\'e inequality (\ref{Poincare inequality}). Then (\ref{Poisson equation}) has a weak solution.
\end{theorem}
\begin{proof}
	We are going to use a classic variational direct method, based on Poincar\'e inequality.
	Let $I: D(\calE) \to \R$ be defined as $I(u):=\frac{1}{2}\|\partial u\|_\calH^2-\Re\langle f,u \rangle_{L^2(\tau)}$.
	Note that for any $\delta>0$
	$$
	I(u)\geq \frac{1}{2}\left(\|\partial u\|_\calH^2-\delta\|u\|_{L^2(\tau)}^2-\frac{1}{4\delta}\|f\|_{L^2(\tau)}^2\right).
	$$
	Hence it follows from Poincar\'e inequality that $I(u)$ has a finite infimum $I_*$.
	
	Let $(u_n)_{n=1}^\infty \subset D(\calE)$ be such that $I(u_n)\searrow I_*$. We need to show that there is a subsequence converging to a limit in $D(\calE)$.
	
	By H\"older and Poincar\'e inequalities,
	\begin{equation}
	I(u)\geq \frac{1}{2}\|\partial u \|^2_\calH - \|u\|_{L^2(\tau)}\|f\|_{L^2(\tau)}\geq\frac{1}{2}\|\partial u \|^2_\calH - C\|\partial u \|_\calH= \left(\frac{1}{2}\|\partial u \|_\calH-C\right)\|\partial u \|^2_\calH.
	\end{equation}
	Hence, the uniform bound on $I(u_n)$ implies uniform boundedness of  $\|\partial u \|_\calH$, and, hence, uniform boundedness of $(u_n)$ in $D(\calE)$. Thus $(u_n)$ has a weakly convergent subsequence, denoted in the same way by $(u_n)$. Denote its limit by $u_\infty$.
	
	Since $I(u)$ is a convex functional,
	$$
	I(u_\infty)\leq \liminf_n I(u_n)=I_*.
	$$
	Thus $u_\infty$ is a minimizer.
	
	Let $\omega$ be some element of $D(\calE)$. Consider a real function $\alpha \mapsto I(u_\infty +\alpha \omega)$. Since
	$I(u_\infty +\alpha \omega)= \alpha^2\frac{1}{2}\calE[\omega] +\alpha\Re\left(\calE(u_\infty,\omega)-\langle f, \omega\rangle_{L^2(\tau)}\right)+I[u_\infty]$, the value of 
	$$
	\alpha^2\frac{1}{2}\calE[\omega] +\alpha\Re\left(\calE(u_\infty,\omega)-\langle f, \omega\rangle_{L^2(\tau)}\right)
	$$ 
	should be nonnegative for any $\alpha\in \R$. Hence $u_\infty$ satisfies 
	$\Re\langle \partial u_\infty, \partial g \rangle_\calH = \Re\langle f, g \rangle_{L^2(\tau)}$,
	and, since $\Re(\alpha z)=0,\; \forall \alpha\in \C$ implies $z=0$, it is a solution of (\ref{Poisson equation}).
\end{proof}

Let us define quasilinear elliptic equation as an equation of the form:
\begin{equation}
\label{quasilinear elliptic equation}
\partial^*F(\partial u)= f,\; f\in D(\calE),
\end{equation}
where $F:\calH \to \calH$ is some (possibly nonlinear) map.

\begin{definition}
	An element $u\in D(\calE)$ is called a weak solution for (\ref{quasilinear elliptic equation}) iff
	$$
	\langle F(\partial u), \partial g \rangle_\calH = \langle f, g \rangle_{L^2(\tau)}
	$$
	for any $g\in D(\calE)$.
\end{definition}

The following theorem is an analogue of a classic Browder-Minty result.
\begin{theorem}
Assume that $F:\calH \to \calH$ satisfies the following properties:
\begin{enumerate}
	\item $\Re\langle F(h) - F(v), h-v \rangle_\calH\geq 0$, for any $h,v\in \calH$,
	\item $\|F(h)\|_\calH\leq c_0 (1+\|h\|_\calH)$, for any $h\in \calH$,
	\item $\Re\langle F(h), h \rangle_\calH\geq c_1\|h\|_\calH-c_2$, for any $h\in \calH$,
\end{enumerate}
$c_0, c_1>0$, $c_2 \geq 0$ and $\calE$ satisfies Poincar\'e inequality (\ref{Poincare inequality}). Then 
the equation (\ref{quasilinear elliptic equation}) has a weak solution.
\end{theorem}
\begin{proof}	
	Let $D_0(\calE)$ be a quotient Hilbert space $D(\calE)/ \ker \calE$, where
	$$
	\ker \calE:=\left\{u\in D(\calE): \calE[u]=0\right\},
	$$
	equipped with the inner product $\calE(\cdot, \cdot)$. It is clear that the equation (\ref{quasilinear elliptic equation}) is well defined on $D_0(\calE)$, and if it has a solution $u\in D_0(\calE)$ then each representative $\tilde{u}\in D(\calE)$ of $u$ appears to be a solution of (\ref{quasilinear elliptic equation}) in $D(\calE)$.
	
	Consider an orthonormal basis $\{\omega_n\}$ of $D_0(\calE)$. We will look for a sequence of elements $(u_m)\subset D_0(\calE)$ that has the form:
	\begin{equation}
	\label{Galerkin approximation}
	u_m=\sum_{k=1}^m d^k_m \omega_k,
	\end{equation}
	such that
	\begin{equation}
	\label{Projection of the problem}
	\Re\langle F(\partial(u_m)), \partial(\omega_k) \rangle_\calH=\Re\langle f, \omega_k \rangle_{L^2(\tau)}.
	\end{equation}
	
	Define the function $V: \R^m \to \R^m$, $V=(v^1,\dots,v^m)$ as follows:
	$$
	v^k(d):=\Re\left\langle F\left(\sum_{j=1}^m d_j \partial\omega_j\right), \partial(\omega_k) \right\rangle_\calH-\Re\langle f, \omega_k \rangle_{L^2(\tau)},
	$$
	where $d=(d_1,\dots,d_m)\in \R^m$. Then
	\begin{align*}
	V(d)\cdot d &= \Re\left\langle \partial\left(\sum_{j=1}^m d_j \partial\omega_j\right), \left(\sum_{j=1}^m d_j \partial(\omega_j)\right) \right\rangle_\calH-\Re\left\langle f, \left(\sum_{j=1}^m d_j \omega_j\right) \right\rangle_{L^2(\tau)}\\
	&\geq c_1\left\|\sum_{j=1}^m d_j \partial\omega_j\right\|^2_\calH-c_2-\Re\left\langle f, \left(\sum_{j=1}^m d_j \omega_j\right) \right\rangle_{L^2(\tau)}\\
	&=c_1|d|^2-c_2-\sum_{j=1}^m d_j \Re\langle f, \omega_j\rangle_{L^2(\tau)}\\
	&\geq \frac{c_1}{2}|d|^2 - c_2 - C\sum_{j=1}^m \Re\langle f, \omega_j\rangle_{L^2(\tau)}^2.
	\end{align*}
	
	Let $u\in D_0(\calE)$ be a equivalence class of a weak solution for (\ref{Poisson equation}). Then
	$$
	\Re\langle \partial u, \partial \omega_j \rangle_\calH = \Re\langle f, \omega_j \rangle_{L^2(\tau)},\; \forall j\in \N,
	$$
	Hence
	$$
	\sum_{j=1}^m (\Re\langle f, \omega_j \rangle_{L^2(\tau)})^2=\Re\calE(u, \omega_j)^2\leq \|u\|_\calH^2\leq C\|f\|_{L^2(\tau)},
	$$
	and $V(d)\cdot d\geq \frac{c_1}{2}|d|^2-C$ for some $C\in \R$. Thus $V(d)\cdot d\geq 0$ if $|d|$ is large enough.
	
	One can use a standard application of Brouwer fixed-point theorem (see p.496 of \cite{Evans}) to conclude that $V(d)=0$ for some point $d\in \R^m$. This point provides us with a ``correct'' set of coefficients in (\ref{Galerkin approximation}) to solve (\ref{Projection of the problem}).
	
	Since
	$$
	\Re\langle F(\partial u_m), \partial u_m \rangle_\calH = \Re\langle f, u_m \rangle_{L^2(\tau)},\; \forall m\in \N,	
	$$
	$$
	c_1 \|\partial u_m\|^2 \leq C + \Re\langle f, u_m \rangle_{L^2(\tau)}\leq C+ \delta \|u_m \|^2_{L^2(\tau)}+\frac{1}{4\delta}\|f\|^2_{L^2(\tau)} 
	$$
	for any $\delta>0$, and choosing $\delta$ small enough, one deduces:
	$$
	\calE[u_m]\leq C(1+\|f\|_{L^2(\tau)}).
	$$
	
	Due to this uniform boundedness, one can extract a subsequence of $(u_m)\subset D_0(\calE)$ (denoted again by $(u_m)$) weakly convergent in $D_0(\calE)$ to some element $u_\infty$, and, due to Poincar\'e inequality, there is a sequence  $(\tilde{u}_k)\subset D(\calE)$ of representatives, $\tilde{u}_k\in u_k$, $\forall k\in N$, such that it converges in $D(\calE)$ to $\tilde{u}_\infty\in u_\infty$.
	
	Using the assumption (2) of the theorem one obtains that there is a subsequence $(u_m)$ such that $F(\partial u_m)\to \xi$ weakly in $\calH$,
	$$
	\Re\langle \xi, \partial \omega_m \rangle_\calH = \Re\langle f, \omega_m \rangle_{L^2(\tau)}.
	$$
	Thus
	$$
	\Re\langle \xi, \partial \omega \rangle_\calH = \Re\langle f, \omega \rangle_{L^2(\tau)},\; \forall \omega\in D_0(\calE).
	$$
	
	And, since by the assumption (1) of the theorem, $\Re\langle F(\partial u_m) - F(\partial \omega), \partial u_m-\partial \omega \rangle_\calH\geq 0$ for any $\omega\in D_0(\calE)$, one can conclude
	\begin{align*}
		\Re\langle f, u_m \rangle_{L^2(\tau)} - \Re\langle F(\partial u_m), \partial \omega \rangle_\calH-
		\Re\langle F(\partial \omega), \partial u_m -\partial \omega \rangle_\calH\geq 0,\\
		\Re\langle f, u_\infty \rangle_{L^2(\tau)} - \Re\langle \xi, \partial \omega \rangle_\calH-
		\Re\langle F(\partial \omega), \partial u_\infty -\partial \omega \rangle_\calH\geq 0.
	\end{align*}
	It follows that
	$$
	\Re\langle \xi- F(\partial \omega), \partial (u_\infty-\omega) \rangle_\calH\geq 0,\; \forall \omega\in D_0(\calE).	
	$$
	Fix any $g\in D_0(\calE)$ and set $\omega:=u_\infty-\lambda g$. One obtains
	\begin{align*}
	\Re\langle (\xi- F(\partial u_\infty - \lambda \partial g)), \partial g \rangle_\calH&\geq 0,\\
	\Re\langle (\xi- F(\partial u_\infty)), \partial g \rangle_\calH&\geq 0,\; \forall g\in D_0(\calE).			
	\end{align*}
	Replacing $g$ by $-g$, one can deduce, that the equality
	$
	\Re\langle (\xi- F(\partial u_\infty)), \partial g \rangle_\calH= 0,\; \forall g\in D_0(\calE)
	$
	holds. Since $
	\Re\langle \xi, \partial \omega \rangle_\calH = \Re\langle f, \omega \rangle_{L^2(\tau)},\; \forall \omega\in D_0(\calE)
	$,
	$u_\infty$ is a weak solution of (\ref{quasilinear elliptic equation}).
	Since $\Re(\alpha z)=0,\; \forall \alpha\in \C$ implies $z=0$, a solution of
	$$
	\Re\langle F(\partial u), \partial g \rangle_\calH = \Re\langle f, g \rangle_{L^2(\tau)}
	$$
	is a solution of (\ref{quasilinear elliptic equation}).	
\end{proof}

\begin{theorem}
Assume that $F:\calH \to \calH$ satisfies the property:
$$
\Re\langle F(h) - F(v), h-v \rangle_\calH\geq \theta\|h-v\|_\calH^2,
$$
for any $h,v\in \calH$, $\theta>0$.
Then the equation (\ref{quasilinear elliptic equation}) has at most one solution in $D_0(\calE)$.
\end{theorem}
\begin{proof}
	Assume that $u_1$ and $u_2$ are both weak solutions of (\ref{quasilinear elliptic equation}), i.e.
	$$
	\langle F(\partial u_1), \partial g \rangle_\calH=\langle F(\partial u_2), \partial g \rangle_\calH = \langle f, g \rangle_{L^2(\tau)}
	$$
	Thus
	$$
	\Re\langle F(\partial u_1)-F(\partial u_2), \partial g \rangle_\calH=0, \; \forall g\in D(\calE).
	$$	
	By the assumption of the theorem,
	$$
	0=\Re\langle F(\partial u_1)-F(\partial u_2), \partial u_1-\partial u_2 \rangle_\calH \geq \theta \|\partial u_1-\partial u_2\|_\calH^2\geq 0.
	$$
\end{proof}

\section{Linear evolution equations}
One can consider in our setting not only elliptic, but also evolution equations, e.g. continuity equation. To establish an existence result it is convenient to use the Bochner space formulation of equations and solutions.

Let $V$ be a Banach space, $p\in[1,\infty)$
$$
L^p([0,T];V):=\{[u]: u:[0,T]\to V, \int_{0}^{T} \|u\|^p dt < \infty\},
$$
where the integral is understood in the sense of Bochner integration, $[u]$ is an equivalence class of a.e. coinciding functions. The norm on $L^p([0,T];V)$ is defined as follows $\|u\|^p_{L^p([0,T];V)}:=\int_{0}^{T}\|u(t)\|_V^p dt$.

The pair $j: V\to H$, $j^*: H \to V^*$ is called Gelfand triple iff $V$ is a Banach space, $H$ is a Hilbert space, $V^*$ is a Banach dual of $B$, the map $j:V\to H$ is linear, continuous, injective, with dense image, and $j^*:H\to V$ is defined by
$$
j^*(h)(v) = \langle h, j(v) \rangle_H,
$$
It follows that $j^*$ is also linear, continuous, injective and with dense image.

\begin{example}
	\label{noncom triple}
	Assume that a Dirichlet space $(\calN, \tau, \calE, D(\calE))$ and a regular $C^*$-algebra $\calA$ are chosen s.t. $D(\calE)$ is a separable space.
	Then $V:=D(\calE)$ equipped with the inner product $\langle \cdot, \cdot \rangle_{\calE}:=\langle \cdot, \cdot \rangle_{L^2(\tau)}+\calE(\cdot, \cdot)$, $H:=L^2(\tau)$, $j(f):=f$, is an example of a Gelfand triple.
\end{example}

\begin{definition}[Weak derivative]
	Let $u\in L^2([0,T];V)$. An element $\omega\in L^2([0,T],V^*)$ is called a weak derivative iff
	$$
	\int_{0}^{T} \bar{\omega}(t) \phi(t) dt = - \int_{0}^{T} \bar{u}(t) \dot{\phi}(t) dt,
	$$
	for any $\phi\in C_0^\infty([0,T],\C)$. The integrals are understood in Bochner sense. 
\end{definition}
It is standard that if a weak derivative exists it is unique (see e.g. Proposition 11.10 in \cite{Brok16}). We denote it by $\dot{u}$.

\begin{theorem}[Theorem 11.13 of \cite{Brok16}]
	\label{existance of solution for evolution equation}
	Let $j: V\to H$, $j^*: H \to V^*$ be a Gelfand triple, where $V$ is a separable infinite-dimensional Banach space. If
	$F: V\times V \times [0,T] \to \R$ is a bilinear form for every $t\in [0,T]$ satisfying
	\begin{enumerate}
		\item $F(v,v;t)\geq c_0 \|v\|^2_V - c_1\|v\|_H^2$, 
		\item $|F(v,w;t)|\leq c_2 \|v\|_V\|w\|_V$.
	\end{enumerate}
	for any $v,w\in V$, $t\in [0,\infty]$, some $c_0, c_1, c_2>0$, which is Lebesgue measurable in $t$.
	Then there is a function $u\in L^2([0,T],V)$ such that $\dot{u}\in L^2([0,T],V^*)$ and it satisfies the equation
	\begin{align}
	\label{Evolution equation}
	\dot{u}(t)(v) + F(u(t),v;t)&=b(t)(v),\; \forall v\in V,\\
	u(0)&=u_0
	\end{align}
	for every $t\in [0,T]$. Here $u_0\in V$, $b\in L^2([0,T];V^*)$.
\end{theorem}

\begin{definition}[Continuity equation]
	\label{Continuity equation}
	Let $V:=D(\calE)$, $H:=L^2(\tau)$,	$j(f):=f$ as in the Example (\ref{noncom triple}). A continuity equation is an equation of the form (\ref{Evolution equation}) with
	$$
	F(u,v,t):=\Re\langle h(t) u, \partial v\rangle_\calH,
	$$
	where $h: [0,T]\to \calH$ is a measurable flow of vector fields, $\calH$, $\partial$ are the elements of the symmetric differential calculus associated with $\calE$. A flow $b\in L^2([0,T]; D(\calE)^*)$ can be interpreted as a source/sink that depends on time.
\end{definition}

\begin{corollary}
	If $h: [0,T]\to \calH$ is such that for any $u\in \calB:=\calD(\calE)\cap\calA$
	\begin{enumerate}
		\item $\Re\langle h(t) u, \partial u\rangle_\calH \geq c_0 \calE[u] - c_1\|u\|_{L^2(\tau)}^2$, 
		\item $|\Re\langle h(t) u, \partial \omega\rangle_\calH|\leq c_2 \|u_1\|_\calE \|u_2\|_\calE$,
	\end{enumerate}
	where $c_0, c_2>0$, $c_1\in \R$, $\|\cdot\|_\calE:=(\calE[\cdot]+\|\cdot\|_{L^2(\tau)}^2)^{\frac{1}{2}}$, then there is a solution of (\ref{Continuity equation}) in the sense of Theorem (\ref{existance of solution for evolution equation}).
\end{corollary}


\begin{thebibliography}{100}

	\bibitem{Alb77}
	Sergio Albeverio, Raphael Hoegh-Krohn,
	\emph{\href{http://link.springer.com/content/pdf/10.1007\%2FBF01611502.pdf}{Dirichlet Forms and Markov Semigroups on $C^*$-Algebras}}, 1977.		
	
	\bibitem{Brok16}
	Martin Brokate,
	\emph{\href{http://arxiv.org/abs/1307.4818}{Lecture notes. Partial Differential Equations 2. Variational Methods}}, 2016.

	\bibitem{BHV07}
	B. Bekka, P. de la Harpe and A. Valette
	\emph{\href{https://perso.univ-rennes1.fr/bachir.bekka/KazhdanTotal.pdf}{Kazhdan’s Property (T)}}, 2007.
		
	\bibitem{Cip98}
	Fabio Cipriani,
	\emph{\href{http://matwbn.icm.edu.pl/ksiazki/bcp/bcp43/bcp43113.pdf}{The variational approach to the Dirichlet problem in $C^*$-algebras}}, 1998.	
	
	\bibitem{Cip08}
	Fabio Cipriani,
	\emph{\href{https://www.researchgate.net/publication/225354326_Dirichlet_Forms_on_Noncommutative_Spaces}{Dirichlet Forms on Noncommutative Spaces}}, 2008.

	\bibitem{Cip14}
	Fabio Cipriani
	\emph{\href{http://www.mat.uniroma2.it/~mp/2014NGA/Slides/Cipriani/}{Noncommutative Potential Theory, (slides)}}, 2014.
	
	\bibitem{CS03}
	Fabio Cipriani, Jean-Luc Sauvageot,
	\emph{\href{https://www.researchgate.net/publication/243023711_Derivations_as_square_roots_of_Dirichlet_forms}{Derivations as square roots of Dirichlet forms}}, 2003.	
	
	\bibitem{CS03-2}
	Fabio Cipriani, Jean-Luc Sauvageot,
	\emph{\href{http://link.springer.com/content/pdf/10.1007/s00039-003-0421-z.pdf}{Noncommutative potential theory and the sign of the curvature operator in Riemannian geometry}}, 2003.

	\bibitem{CS07}
	Fabio Cipriani, Jean-Luc Sauvageot,
	\emph{\href{http://arxiv.org/abs/0707.0840}{Fredholm Modules on P.C.F. Self-Similar Fractals and their Conformal Geometry}}, 2007.
	
	\bibitem{CS12}
	Fabio Cipriani, Jean-Luc Sauvageot,
	\emph{\href{http://arxiv.org/abs/1207.3524}{Variations in noncommutative potential theory: finite energy states, potentials and multipliers}}, 2012.
	
	\bibitem{CFK12}
	Fabio Cipriani, Uwe Franz, Anna Kula,
	\emph{\href{http://arxiv.org/abs/1210.6768v2}{Symmetries of Lévy processes on compact quantum groups, their Markov semigroups and potential theory}}, 2012.

	\bibitem{DL92}
	E. Brian Davies, J. Martin Lindsay
	\emph{\href{http://link.springer.com/article/10.1007\%2FBF02571804}{Non-commutative symmetric Markov semigroups}}, 1992.

	\bibitem{Evans}
	Lawrence C. Evans, \emph{Partial Differential Equations: Second Edition}, Graduate Series in Mathematics, vol. 19.R, 2010.

	\bibitem{IRT11}
	Marius Ionescu, Luke G. Rogers, and Alexander Teplyaev,
	\emph{\href{http://www.math.uconn.edu/~rogers/Preprints/derivations.pdf}{Derivations, Dirichlet forms and Spectral analysis}}, 2011.

	\bibitem{IRT11-2}
	Marius Ionescu, Luke G. Rogers, and Alexander Teplyaev,
	\emph{\href{http://arxiv.org/abs/1106.1450}{Derivations and Dirichlet forms on fractals}}, 2011.
	
	\bibitem{Helem}
	A. Ya. Helemskii,
	\emph{Quantum Functional Analysis: Non-Coordinate Approach}, University Lecture Series
	Volume: 56; 2010.	
		
	
	\bibitem{Hinz12}
	Michael Hinz, Alexander Teplyaev
	\emph{\href{https://arxiv.org/abs/1207.6375v2}{Vector analysis on fractals and applications}}, 2012.

	\bibitem{Hinz13}
	Michael Hinz, Michael R\''ockner, Alexander Teplyaev,
	\emph{\href{https://arxiv.org/abs/1202.0743v6}{Vector analysis for Dirichlet forms and quasilinear PDE and SPDE on metric measure spaces}}, 2013.
	
	\bibitem{Hinz13-2}
	Michael Hinz, Alexander Teplyaev
	\emph{\href{https://arxiv.org/abs/1206.6644v3}{Local Dirichlet forms, Hodge theory, and the Navier-Stokes equations on topologically one-dimensional fractals}}, 2013.

	\bibitem{Holevo}
	Alexander S. Holevo, 
	\emph{\href{https://arxiv.org/abs/quant-ph/9701037}{Covariant Quantum Dynamical Semigroups: Unbounded generators}}, 1997.	
	
	\bibitem{JunZeng13}
	Marius Junge, Qiang Zeng,
	\emph{\href{http://arxiv.org/abs/1211.3209}{Noncommutative martingale deviation and Poincar\'e type inequalities with applications}}, 2013.
	
	\bibitem{JunZeng13-2}
	Marius Junge, Qiang Zeng,
	\emph{\href{http://arxiv.org/abs/1311.5098}{Subgaussian 1-cocycles on discrete groups}}, 2013.

	\bibitem{Kos13}
	Ryszard Paweł Kostecki,
	\emph{\href{http://arxiv.org/abs/1307.4818}{W*-algebras and noncommutative integration}}, 2013.

	\bibitem{Petz88}
	Dénes Petz,
	\emph{\href{https://projecteuclid.org/euclid.cmp/1104160586}{A variational expression for the relative entropy}}, 1988.

	\bibitem{Rieffel}
	Dénes Petz,
	\emph{\href{https://arxiv.org/abs/1401.4622}{Non-Commutative Resistance Networks}}, 2014.

	\bibitem{Sav96}
	Jean-Luc Sauvageot,
	\emph{\href{https://www.researchgate.net/publication/243023263_Semi-groupe_de_la_chaleur_transverse_sur_la_C*-algbre_d'un_feuilletage_riemannien}{Semi-groupe de la chaleur transverse sur la C*-algèbre d'un feuilletage riemannien}}, 1996.
	
	\bibitem{Sav99}
	Jean-Luc Sauvageot,
	\emph{\href{https://www.researchgate.net/publication/255636445_STRONG_FELLER_SEMIGROUPS_ON_C_-ALGEBRAS}{Strong Feller semigroups on $C^*$-algebras}}, 1999.
	
	\bibitem{Sav07}
	Jean-Luc Sauvageot,
	\emph{\href{http://www.math.ucla.edu/~shlyakht/berkeley-07/conference/contrib/sauvageot-talk.pdf}{Dirichlet Forms on $C^*$-algebras. A Review}}, 2007.
	
	\bibitem{Xu07}
	Qunahua Xu,
	\emph{\href{http://www.nim.nankai.edu.cn/activites/conferences/hy200707/summer-school/Xu-long.pdf}{Operator spaces and noncommutative $L^p$}}, 2007.	
	
	\bibitem{Zeng13}
	Qiang Zeng,
	\emph{\href{http://arxiv.org/abs/1306.6099v2}{Poincar\'e type inequalities for group measure spaces and related transportation cost inequalities}}, 2013.
	
	\bibitem{Zeng14}
	Qiang Zeng,
	\emph{\href{https://www.ideals.illinois.edu/handle/2142/50527}{Poincar\'e inequalities in noncommutative $L^p$ spaces, (thesis)}}, 2014.

\end{thebibliography}
\end{document}